\documentclass[12pt]{article}
\usepackage[T2A]{fontenc}
\usepackage[utf8]{inputenc}
\usepackage{xcolor}
\usepackage{amssymb,amsfonts,amsmath,mathtext,amsthm}
\usepackage{cite,enumerate,float,indentfirst}
\usepackage{graphics,epsfig,graphicx,epstopdf,amscd,latexsym,verbatim,hyperref,ulem}
\usepackage[top = 30mm, bottom = 20mm, left = 30mm, right = 25mm]{geometry}

\newtheorem{lemma}{Lemma}[section]
\newtheorem{example}{Example}[section]
\newtheorem{theorem}{Theorem}[section]
\newtheorem{corollary}{Corollary}[section]

\newtheorem{definition}{Definition}[section]
\newtheorem{remark}{Remark}[section]
\newtheorem{problem}{Problem}[section]

\def\id{\mathrm{id}}

\def\Imm{\mathrm{Im}\,}
\def\Aut{\mathrm{Aut}}

\def\Span{\mathrm{Span}}

\begin{document}

\begin{center}
{\Large Monomial Rota---Baxter operators of nonzero weight on $F[x,y]$ \\
coming from averaging operators}

A. Khodzitskii
\end{center}

\begin{abstract}
The intensive study of Rota---Baxter operators on the polynomial algebra $F[x]$ has been started with the work of S.H. Zheng, L. Guo, and M. Rosenkranz (2015).
We deal with the case of two variables and monomial Rota---Baxter operators of nonzero weight. 
The family of such operators arisen from homomorphic averaging operators on $F[x,y]$ is described.

\end{abstract}

\section{Introduction}

Rota---Baxter operator is an algebraic generalization of the integral operator.
Given an algebra $A$ over a field $F$, a linear operator $R$ on $A$ is called a Rota---Baxter operator (RB-operator, for short), if the following relation 
$$
R(a)R(b) = R\big(R(a)b + a R(b) + \lambda a b\big)
$$
holds for all $a,b \in A$.
Here $\lambda\in F$ is a fixed scalar called a weight of $R$.

When $\lambda = 0$, the relation is a generalization of the integration by parts formula. 
The definition of Rota---Baxter operator was given by G. Baxter in 1960~\cite{Baxter}. 
However, firstly the defining relation for nonzero $\lambda$ appeared in the work of F. Tricomi in 1951~\cite{Tricomi}.

G.-C. Rota was one of the most active researchers in the area of different algebraic operators, including Rota---Baxter operators then called as Baxter ones, see, e.\,g.~\cite{Rota,Rota2}.
In 1980s, Rota---Baxter operators were rediscovered in the direction of the classical and modified Yang---Baxter equations from mathematical physics~\cite{BelaDrin82,Semenov83}.
Now, we observe the growing interest in the study of RB-operators; 
in 2012, Li Guo wrote a monograph on the subject~\cite{GuoMonograph}.

The study of Rota---Baxter operators on the polynomial algebra has began only recently, in 2015, by S. H. Zheng, L. Guo, and M. Rosenkranz~\cite{Monom2}. 
In that work, the authors introduced so called monomial RB-operators, i.\,e. such RB-operators that map every monomial to either 0 or to some monomial with a nonzero coefficient.
S. H. Zheng, L. Guo, and M. Rosenkranz described all injective monomial RB-operators of weight~0 on $F[x]$ over a field~$F$ of characteristic zero.

In 2016, H. Yu classified all monomial RB-operators on $F[x]$ over a field~$F$ of characteristic zero~\cite{Monom}. 
In 2020, the description of monomial RB-operators on the ideal $F_0[x] = \langle x\rangle$ in $F[x]$ over a field~$F$ of characteristic zero was obtained~\cite{MonomNonunital}.
In~\cite{ReprRB,ReprRB2,ReprRB3,ReprRB4}, irreducible and indecomposable finite-dimensional representations of the corresponding Rota---Baxter algebras were characterized.

In~\cite{GubPer}, the conjecture of S. H. Zheng, L. Guo, and M. Rosenkranz was confirmed and thus, injective RB-operators of weight~0 on $F[x]$, where $\mathrm{char}\,F = 0$, were classified.  

Let us note the works~\cite{Ogievetsky,Viellard-Baron}, where examples of Rota---Baxter operators on $F[x,y]$ appeared.

For formal inverses of Rota---Baxter operators, derivations, monomial operators have been investigated on the polynomial algebra and the field of rational functions in a~series of works~\cite{NowZiel,Ollagnier,Kitazawa,Essen}.
Monomial derivations were used to construct counterexamples to the 14th problem of Hilbert.

In the current work, we initiate a study of monomial Rota---Baxter operators of nonzero weight on $F[x,y]$. We deal with the family of such operators which are induced by monomial homomorphic averaging operators on $F[x,y]$.
The crucial fact is that every homomorphic averaging operator $T$ on an algebra $A$ is automatically an RB-operator of weight~$-1$ on~$A$ (Lemma~\ref{R=-T_Lemma}).
Describing monomial homomorphic averaging operators on $F[x,y]$, we obtain in Theorem~\ref{AveOpClassif} (except simple cases) the following operators:

(1) $T(x^n y^m) = x^{r m} y^m$, $r \in \mathbb{N}$, 

(2) $T(x^n y^m) = \alpha^n y^{m + r n}$, $r \in \mathbb{N}$, $\alpha\in F^*$.

Actually, up to the choice of scalars, monomial homomorphic averaging operators on $F[x_1,\ldots,x_n]$ occur to be in one-to-one correspondence with idempotent matrices from~$M_n(\mathbb{N})$.

Now, we formulate the main problem of the work. 

{\bf Problem 1}. 
Describe all Rota---Baxter operators of nonzero weight on $F[x,y]$ of the form 
$$
R(x^n y^m) = \alpha_{n,m}T(x^n y^m),
$$
where $\alpha_{n,m}$ are some scalars from~$F$ 
and $T$ is a monomial homomorphic averaging operator on $F[x,y]$.

We solve this problem completely for the case (1) and for (2), when $r = 0,1$.
In the case (2), $r > 1$, the problem is much more difficult.
However, we expect that all difficulties are concerned only with technical details and the approach applied for $r = 1$ works for $r>1$ too.

In Theorem~\ref{Thm:decompositions}, all decompositions of $F[x,y]$ into a direct sum of two subalgebras containing a basis of monomials are described.
Every such decomposition corresponds to a monomial RB-operator of nonzero weight on $F[x,y]$ satisfying the condition that either $x^n y^m\in \ker R$ or $x^n y^m\in \Imm R$ for every monomial $x^n y^m$.

Let us give a short outline of the work.
In~\S2, we provide required preliminaries on Rota---Baxter and averaging operators.
In~\S3, we classify all monomial homomorphic averaging operators on $F[x,y]$ (Theorem~\ref{AveOpClassif}).
In~\S4 and in~\S5, we solve Problem~1 in simple cases (3)--(5).

In~\S6, we clarify everything in the case (1) for any value of~$r$.

In~\S7, we solve Problem in the case (2), when $r = 1$. We omit the case (2), when $r=0$, since it coincides with the case (1), $r = 0$. 
The method which we apply to classify all RB-operators coming from the case (2), $r = 1$ is following. By induction, all coefficients $\alpha_{n,m}$ for $n\geq2$ may be expressed via $\alpha_{1,k}$. Thus, the main difficulty is to understand how we may express $\alpha_{1,k}$ through $\alpha_{1,l}$ for $l<k$.
For this, we solve Problem in the particular cases $\alpha_{1,0} = \alpha_{1,1}$ (Lemma~\ref{lem:q0=q1}), $\alpha_{1,0} = \alpha_{1,2}$ (Lemma~\ref{lem:q0=q2}),
and $\alpha_{1,1} = \alpha_{1,2}$ (Lemma~\ref{lem:q1=q2}).
Further, we change coefficients by the formula
$$
s_i = \alpha - \beta \alpha_{1,i}, \quad
\alpha = \frac{\alpha_{1,1}}{\alpha_{1,0}-\alpha_{1,1}},\
\beta = \frac{1}{\alpha_{1,0}-\alpha_{1,1}}.
$$
We prove that $s_i$ satisfy the recurrence
$s_{m+1} = \dfrac{s_m + 1}{1/s_2 - s_m}$, $m\geq2$, and we solve it.

Finally, we obtain the main result of the work, it is Theorem~\ref{theorem_r=1}.
We show that any monomial operator on $F[x,y]$ of the form $R(x^n y^m) = \alpha_{n,m}y^{n+m}$ may be uniquely extended from the values $\alpha_{1,0},\alpha_{1,1},\alpha_{1,2}$ to an RB-operator.
The additional condition $2\alpha_{1,1}\neq \alpha_{1,0}+\alpha_{1,2}$ is technical, see Remark~\ref{rem:s2}.

{\bf Theorem}.
Let $\alpha_{1,0},\alpha_{1,1},\alpha_{1,2}\in F$ be pairwise distinct such that $2\alpha_{1,1}\neq\alpha_{1,0}+\alpha_{1,2}$.
Then a monomial operator of the form $R(x^n y^m) = \alpha_{n,m}y^{n+m}$, where $\alpha_{n,m}$ are expressed via $\alpha_{1,0},\alpha_{1,1},\alpha_{1,2}$ by ~\eqref{2_1n},~\eqref{k+1n}, is an RB-operator of weight $1$ on $F[x,y]$.

In~\S8, we show how to apply obtained results to get RB-operators of nonzero weight on the algebra of the Laurent polynomials $F[x,x^{-1},y,y^{-1}]$ and the algebra of formal power series $F[[x,y]]$.

We follow some conventions.
First, throughout the work, we suppose that $F$ is field of characteristic 0. 
Second, given a~scalar~$\mu\in F$, we assume that $\mu^0 = 1$.
Second, for any scalar~$\mu\in F$, we assume that $\mu^0 = 1$, in particular, $0^0 = 1$.

\section{Preliminaries}

\begin{definition}
Let $A$ be an algebra over a field $F$ and $R$ be a~linear operator on $A$. 
Then $R$~is called a Rota---Baxter operator (RB-operator, for short)
if for all $a,b\in A$ the following relation holds,
\begin{gather} \label{RBO}
R(a)R(b) = R\big(R(a)b + a R(b) + \lambda a b\big),
\end{gather}
where $\lambda\in F$ is a fixed constant (a weight of $R$).
\end{definition}

Within the work, we consider RB-operators only of nonzero weight.
Given an algebra~$A$ and an RB-operator of weight~$\lambda\neq0$ on $A$, both $\ker R$ and $\Imm R$ are subalgebras of~$A$.

Given an algebra~$A$, below by $A = B\oplus C$ we mean that $A$ is a direct vector-space sum of its two subalgebras $B$ and $C$.

\begin{lemma}[\!\!\cite{GuoMonograph}] \label{SplitRB}
Assume that an algebra $A$ equals $A_1\oplus A_2$,
where $A_1$ and $A_2$ are its subalgebras.
An operator $R$ defined by the rule
\begin{equation}\label{Split}
R(x_1+x_2) = -\lambda x_2,\quad x_1\in A_1,x_2\in A_2,
\end{equation}
is an RB-operator on $A$ of weight $\lambda$.
\end{lemma}

Let us call an RB-operator of nonzero weight~$\lambda$ arisen in Lemma~\ref{SplitRB} as a splitting RB-operator.

\begin{lemma}[\!\!\cite{BGP}] \label{splitting-criterion}
Let $A$ be a unital algebra, and let $R$ be an RB-operator of nonzero weight~$\lambda$ on~$A$.

a) If $R(1)\in F$, then $R$ is splitting.

b) $R$ is splitting if and only if $R(R(x)+\lambda x) = 0$ for all $x\in A$.
\end{lemma}

Let us recall the description of monomial RB-operators of nonzero weight on $F[x]$. 

\begin{theorem}[\!\!\cite{Monom}] \label{F[x]-1}
Each nonzero monomial RB-operator of weight~$-1$ on $F[x]$ is either splitting with
the subalgebras $F$ and $\langle x\rangle$ or has a form $R(x^n) = q^n$ for some $q\in F$.
\end{theorem}

\begin{definition}
A~linear operator~$L$ on $F[x,y]$ is called monomial if for all $k,l\in\mathbb{N}$ we have
$L(x^k y^l) = \varepsilon_{kl}x^{a_{kl}}y^{b_{kl}}$,
where $\varepsilon_{kl}\in F$, $a_{kl},b_{kl}\in\mathbb{N}$. 
\end{definition}

\begin{lemma}[\!\!\cite{MonomNonunital}] \label{lem:monom-split}
Let $R$ be a monomial RB-operator of weight 1 on $F[x,y]$, then
$F[x,y] = \Imm R\oplus \ker R$ and $R|_{\Imm R} = -\id$.
Moreover, $R(1)\in\{0,-1\}$.
\end{lemma}

\begin{definition}[\!\!\cite{AveOp,Pei2015}]
Let $A$ be an algebra, then an operator $T$ on $A$ is called averaging operator if 
$T(a)T(b) = T(T(a)b) = T(a T(b))$ holds for all $a,b\in A$.

An averaging operator $T$ on $A$ is called homomorphic averaging operator if
$T$ is additionally an endomorphism of $A$, i.\,e., $T(ab) = T(a)T(b)$ for all $a,b\in A$.
\end{definition}

The next lemma clarifies the connection between 
Rota --- Baxter operators and homomorphic averaging operators.

\begin{lemma}\label{R=-T_Lemma}
Let $A$ be an algebra and let $T$ be a homomorphic averaging operator on~$A$.
Then $R = - T$ is a Rota --- Baxter operator of weight~1 on $A$.
\end{lemma}

\begin{proof}
We check~\eqref{RBO} directly, 
\begin{multline*}
R(R(a)b + a R(b) + ab)
 = R(-T(a)b - a T(b) + ab)
 = T(T(a)b) + T(a T(b)) - T(ab) \\
 = T(a)T(b)
 = R(a)R(b). 
  \qedhere
\end{multline*} 
\end{proof}

\begin{lemma}\label{ConjAveraging}
Let $A$ be an algebra, let $\psi\in\Aut(A)$, and let $T$ be a~(homomorphic) averaging operator on~$A$.
Then $T^{(\psi)} = \psi^{-1}T\psi$ is a~(homomorphic) averaging operator on~$A$.
\end{lemma}

Define an automorphism $\psi_{x,y}$ of $F[x,y]$ as follows,
$\psi_{x,y}(x) = y$, $\psi_{x,y}(y) = x$.

\section{Monomial homomorphic averaging operators}

By Lemma~\ref{R=-T_Lemma}, every homomorphic averaging operator on $F[x,y]$ provides an RB-operator of weight~1 on $F[x,y]$.
Let us describe all monomial homomorphic averaging operators on $F[x,y]$.

\begin{theorem} \label{AveOpClassif}
Up to conjugation with $\psi_{x,y}$, a nonzero monomial homomorphic averaging operator~$T$ on $F[x,y]$ 
equals one of the following operators:

(1) $T(x^n y^m) = x^{r m} y^m$, $r \in \mathbb{N}$, 

(2) $T(x^n y^m) = \alpha^n y^{m + r n}$, $r \in \mathbb{N}$, $\alpha\in F^*$,

(3) $T(x^n y^m) = x^n y^m$,

(4) $T(x^n y^m) = \begin{cases}
0, & n > 0, \\
y^m, & n = 0,
\end{cases}$

(5) $T(x^n y^m) = \begin{cases}
\alpha^n \beta^m, & n + m > 0, \\
1, & n = m = 0,
\end{cases}$ $\alpha,\beta\in F$.
\end{theorem}

\begin{proof}
Denote $T(x) = \alpha x^a y^c$ and 
$T(y) = \beta x^b y^d$, where $\alpha,\beta\in F$, $a,b,c,d\in\mathbb{N}$. 
Since $T$ is nonzero, we may find $f\in F[x,y]$ such that $T(f)\neq0$. 
Then the equality $T(f) = T(f\cdot1) = T(f)T(1)$ implies that $T(1) = 1$.

If $\alpha = \beta = 0$, then we get the ~operator~5).
Let $\alpha = 0$ and $\beta\neq0$.
If $b\neq0$, then we get a~contradiction, since
$$
T(T(y)\cdot1)
 = \beta T(x^b y^d)
 = 0 \neq
 T(y)T(1)
 = \beta x^b y^d.
$$
So, we get $b = 0$ and $T(y) = \beta y^d$.
From the equality $T^2(y) = T(T(y)\cdot1) = T(y)$ we conclude 
that $d = 1$ and $\beta = 1$, it is the operator~4). 
The case $\alpha\neq0$, $\beta = 0$ after the conjugation with~$\psi_{x,y}$ gives the same operator 4).

Consider $\alpha,\beta\neq0$.
Applying that $T$ is an endomorphism, we get
$T(x^ny^m) = (T(x))^n (T(y))^m = \alpha^n\beta^m x^{an+bm}y^{cn+dm}$.
Now we check that $T$ is an averaging operator,
\begin{multline*}
\alpha^{n+s}\beta^{m+t} x^{a(n+s)+b(m+t)}y^{c(n+s)+d(m+t)}
 = T(x^ny^m)T(x^sy^t) \\
 = T(T(x^ny^m)x^sy^t) 
 = \alpha^n\beta^mT(x^{an+bm+s}y^{cn+dm+t}) \\
 = \alpha^{n + an + s + bm}\beta^{m + cn + dm +t} 
 x^{a(an + s + bm) + b(cn + dm +t)}y^{c(an + s + bm) + d(cn + dm +t)}.
\end{multline*}
Therefore, 
\begin{equation} \label{AveOp:Coeff}
\alpha(\alpha^{an+bm}-1) = \beta(\beta^{cn+dm}-1) = 0, 
\end{equation}
and also
$$
(a-1)(an+bm) + b(cn+dm) = 0, \quad
(d-1)(cn+dm) + c(an+bm) = 0.
$$
Substituting $(n,m) = (0,1)$, $(n,m) = (1,0)$ in the first equation 
and $(s,t) = (0,1)$, $(s,t) = (1,0)$ in the second one, we get equations
\begin{equation} \label{AveOp:deg}
a(a-1) + bc = 0, \quad
d(d-1) + bc = 0, \quad
b(a+d-1) = 0, \quad
c(a+d-1) = 0.
\end{equation}

If $(b,c)\neq(0,0)$, then $a+d = 1$ and either $a = 1, d = 0$, or $a = 0, d = 1$.
In each case we have $bc = 0$. 
Hence, $b = 0$ and $c>0$ is arbitrary or $c = 0$ and $b>0$ is arbitrary.

If $b = c = 0$, then we get four subcases:
1) $a = d = 0$,
2) $a = 0$, $d = 1$,
3)~$a = 1$, $d = 0$,
4) $a = d = 1$.

Up to action of $\psi_{x,y}$, 
we have the following variants for $(a, b, c, d)$:
$(0,0,0,0)$, $(1,0,0,1)$, $(0,r,0,1)$, and $(0,0,r,1)$, where $r\geq0$.
In the first case, it is the operator 5).
In the second and in the third cases 
we get the operators 3) and 1) respectively,
since $\alpha = \beta = 1$ by~\eqref{AveOp:Coeff}. 

In the fourth case, we analogously get that $\beta = 1$ and it is the operator 2).
\end{proof}

\begin{remark}
Let us collect the numbers which we have introduced in the proof of Theorem~\ref{AveOpClassif}
in a matrix $A = \begin{pmatrix}
  a & b\\
  c & d
\end{pmatrix}$.
Forgetting the coefficients at monomials in Theorem~\ref{AveOpClassif},
we get by~\eqref{AveOp:deg} a~correspondence between monomial homomorphic averaging operators on $F[x,y]$ and idempotent matrices from $M_2(\mathbb{N})$.
\end{remark}


\begin{remark} 
We may generalize Theorem~\ref{AveOpClassif} onto $F[x_1,\ldots x_n]$.
As in the case $n = 2$, we introduce parameters $\beta_{ij}\in\mathbb{N}$ and $\alpha_i\in F$ such that $T(x_i) = \alpha_i x^{\beta_{i1}}\ldots x^{\beta_{in}}$, $i\in\{1,\ldots,n\}$. 
Define a homomorphism $\phi_{x_i}$, $i=1,\ldots,n$, acting from the semigroup of monomials to~$F$ as follows,
$\phi_{x_i}(x_j) = \beta_{ij}$, $j=1,\ldots,n$.
We extend $\phi_{x_i}$ on $F[x_1,\ldots,x_n]$ by linearity.
Let $x^{\bar{m}} = x_1^{m_1}\ldots x_n^{m_n}$.
The conditions of $T$~being homomorphic averaging operator imply, if we forget about scalars, the relations
\begin{multline*}
\phi_{x_i}(\bar{k}) + \phi_{x_i}(\bar{t}) 
 = \phi_{x_i}(\bar{k} + \bar{t}) 
 = \phi_{x_i}(k_1 + \phi_{x_1}(\bar{t}),\ldots,k_n+ \phi_{x_n}(\bar{t})) \\
 = \phi_{x_i}(t_1 + \phi_{x_1}(\bar{k}),\ldots, t_n+\phi_{x_n}(\bar{k})),
\end{multline*}
where $\bar{k}=(k_1,\ldots k_n)$, $\bar{t}=(t_1,\ldots t_n)$, $i = 1,\ldots, n$.
In the matrix form, we get the equality $A^2 = A$, where $A = (\beta_{i j})^n_{i, j = 1}\in M_n(\mathbb{N})$.
Thus, monomial homomorphic averaging operators up to the choice of scalars $\alpha_i$ are in one-to-one correspondence with idempotent matrices from $M_n(\mathbb{N})$.
\end{remark}

We state the problem about the description of RB-operators obtained from monomial homomorphic averaging operators.

\begin{problem} 
Find all RB-operators~$R$ of nonzero weight on~$F[x,y]$ satisfying the equality
$$
R(x^n y^m) = \alpha_{n,m} T(x^n y^m)
$$
for some $\alpha_{n,m}\in F$
and a~monomial homomorphic averaging operator~$T$.
\end{problem}

Due to the definition of~$T$, 
we get an infinite system of equations on~$\alpha_{n,m}$.
From Lemma~\ref{R=-T_Lemma}, 
we know that the values $\alpha_{n,m}=-1$ give a solution to this system.

\section{RB-operators coming from Cases (3) and (4)}

In Cases~3 and~4 we deal with an RB-operator $R$ on~$F[x,y]$ of the form $R(x^n y^m) = \alpha_{n,m} x^n y^m$. 
By Lemma~\ref{lem:monom-split}, $\alpha_{n,m} \in \{ 0, -1 \}$ for all $n,m$.
Denote $A = \Imm R$ and $B = \ker R$.
Hence, we study decompositions $F[x,y] = A\oplus B$ of subalgebras,
where both $A$ and $B$ contain bases consisting of monomials.
Suppose that $A,B\neq(0)$.

Consider different cases.

1. Let $x,y$ lie in the same subalgebra, say~$A$. 
Then $B = F1$ and $F[x,y] = F_0[x,y] \oplus F1$.

2. Let $x,y$ lie in different subalgebras, we assume that $x\in A$ и~$y\in B$.

\begin{example}
We have the decompositions
\begin{gather*}
F[x,y]
 = \Span\{x^ky^l\mid k > l\}
 \oplus \Span\{x^ky^l\mid k\leqslant l\}, \\
F[x,y]
 = \Span\{x^ky^l\mid 2k>l\}
 \oplus \Span\{x^ky^l\mid 2k\leqslant l\}
\end{gather*}
of $F[x,y]$ into a direct sum of two subalgebras.
On the other hand,
$$
\Span\{x^ky^l\mid k-l+2>0\}\oplus\Span\{x^ky^l\mid k-l+2\leqslant 0\}
$$
is not a required decomposition,
since the monomial $xy^2$ but not $(xy^2)^2$ satisfies the condition $k-l+2>0$.
\end{example}

2.1. If $B$ contains only monomials $y^n$, $n > 0$, 
or $A$ contains only monomials $x^n$, $n > 0$, then we get the following decompositions
\begin{gather*}
F[x,y] =  \langle x\rangle_{\mathrm{ideal}}\oplus F[y],\quad
F[x,y] = (\langle x\rangle_{\mathrm{ideal}}\oplus F1)\oplus F^{*}[y].\\
F[x,y] =  \langle y\rangle_{\mathrm{ideal}}\oplus F[x],\quad
F[x,y] = (\langle y\rangle_{\mathrm{ideal}}\oplus F1)\oplus F^{*}[x].
\end{gather*}

2.2. Now, we may assume that there exist monomials depending on $y$ in $A$ and there exist monomials depending on~$x$ in~$B$.

Denote
\begin{equation}\label{alpha}
\alpha 
	= \sup\limits_{x^ky^l\in A,\,k,l\geqslant 1}\{l/k\},\quad
\beta 
	= \sup\limits_{x^ky^l\in B,\,k,l\geqslant 1}\{k/l\}.
\end{equation}
The conditions of the case 2.2 imply that $\alpha,\beta<\infty$.
Indeed, suppose that $\alpha = \infty$. 
Consider $x^k y^l\in B$ for $k\geq1$.
Since $\alpha = \infty$, there exists a monomial $x^ay^b\in A$ such that $b/a>l/k$.
On the one hand, $x^{ak}y^{bk}\in A$.
On the other hand, $x^{ak}y^{bk} = (x^k y^l)^a\cdot y^{bk-al}\in B$, a~contradiction.

For $\alpha$, we have the following geometric interpretation.
Given a monomial~$x^k y^l$, we may identify it with a point~$(k,l)$ from the first quadrant of $\mathbb{R}^2$. 
Thus, all monomials from~$A$ satisfy the inequality $\alpha  k\geq l$, i.\,e., they lie below the line $y=\alpha x$ and all monomials from~$B$ satisfy the inequality $k \geq \beta l$ and so, lie above the line $x=\beta y$. 

\begin{lemma}\label{Sh=0_Lemma1}
Let $k,l\geq1$. If $\alpha k > l$, then $x^ky^l\in A$.
If $k>\beta l$, then $x^ky^l\in B$.
\end{lemma}

\begin{proof}
Suppose that there exist a monomial $x^{k_0}y^{l_0}\notin A$ and $\alpha k_0 > l_0 \neq 0$.
Since $l_0/k_0<\alpha$ and due to the definition of the supremum, there exists a~monomial $x^{k_1}y^{l_1} \in A$ such that $l_0/k_0<l_1/k_1<\alpha$.
Suppose $(x^{k_0}y^{l_0})^{l_1}\notin A$, then
\begin{equation*}
(x^{k_0}y^{l_0})^{l_1}=x^{k_0l_1}y^{l_0l_1},\quad 
(x^{k_1}y^{l_1})^{l_0}=x^{k_1l_0}y^{l_1l_0}.
\end{equation*}
The conditions $x \in A$ and $k_1l_0<k_0l_1$ imply that
$$
(x^{k_1}y^{l_1})^{l_0}\cdot x^{k_0l_1-k_1l_0}
 = (x^{k_0}y^{l_0})^{l_1}\in A.
$$
	
Since $x^{k_0}y^{l_0}\notin A$, 
we have $x^{k_0}y^{l_0}\in B$ and $(x^{k_0}y^{l_0})^{l_1}\in B$.
Hence, $(x^{k_0}y^{l_0})^{l_1} \in B \cap A$, which contradicts to $B \cap A=(0)$.
The proof of the second part of Lemma is analogous, up to the action of $\psi_{x,y}$.
\end{proof}

\begin{corollary} \label{Sh=0_Coro}
For the $\alpha$ and $\beta$, the equality $\alpha\beta=1$ holds.
\end{corollary}
	
\begin{proof}
Let $\alpha\beta<1$. 
Then due to the density of $\mathbb{Q}$ in $\mathbb{R}$ there exist $k,l\in\mathbb{N}_{>0}$ such that $\alpha<\frac{l}{k}<\frac{1}{\beta}$.
Geometrically, it means that we may find a~point $(k,l)$ lying in the domain between the lines $y=\alpha x$ and $x=\beta y$.
This point corresponds to a~monomial $x^ky^l$. 
By Lemma~\ref{Sh=0_Lemma1} we get $x^ky^l\notin A$ and $x^ky^l\notin B$, it is a~contradiction.

Let now $\alpha\beta>1$. 
Analogously, we get $\alpha>\frac{l}{k}>\frac{1}{\beta}$ for some $k,l\in\mathbb{N}_{>0}$.
Thus, $x^ky^l$ lies in both subalgebras $A$ and $B$, we again arrive at a~contradiction.
\end{proof}

\begin{lemma}\label{Sh=0_Lemma2}
Let $\alpha\in\mathbb{Q}$.
Then all monomials $x^ky^l$ such that $k,l\geqslant 1$ and $l=\alpha k$ either simultaneously belong to the subalgebra~$A$ or simultaneously belong to the subalgebra~$B$.
\end{lemma}
	
\begin{proof}
Let $\alpha=\frac{l_0}{k_0}$ be an irreducible fraction, without lost of generality we assume that $x^{k_0}y^{l_0}\in A$.
Then for an arbitrary monomial $x^ky^l$ such that $k,l\geqslant 1$ and $l=\alpha k$ we have 
$x^ky^l=(x^{k_0}y^{l_0})^s$ for some $s$.
Hence, $x^ky^l\in A$.
\end{proof}

\begin{remark}\label{Sh=0_remark2}
The parameters $\alpha_1 \neq \alpha_2$ (or $\beta_1 \neq \beta_2$) define different subalgebras.
\end{remark}

\begin{theorem} \label{Thm:decompositions}
All possible decompositions of $F[x,y]$ into a~direct sum of two nonzero subalgebras, each of which contains a monomial basis, up to the action of $\psi_{x,y}$, are listed below,
	
\noindent I) $F[x,y] = F^*[x,y] \oplus F1$, 

\noindent II) $F[x,y] =  \langle x\rangle_{\mathrm{ideal}}\oplus F[y]$,

\noindent III) $F[x,y] = (\langle x\rangle_{\mathrm{ideal}}\oplus F1)\oplus F^{*}[y]$, 

\noindent IV) $F[x,y] = (\Span\{x^ky^l\mid k,l\geq1,\,\alpha k > l\}\oplus F1)\oplus \Span\{x^ky^l\mid k,l\geq1,\,\alpha k\leqslant l\}$, $\alpha\in\mathbb{Q}$,

\noindent V) $F[x,y] = (\Span\{x^ky^l\mid k,l\geq1,\,\alpha k \geq l\}\oplus F1)\oplus \Span\{x^ky^l\mid k,l\geq1,\,\alpha k < l\}$, $\alpha\in\mathbb{Q}$,

\noindent VI) $F[x,y] = (\Span\{x^ky^l\mid k,l\geq1,\,\alpha k > l\}\oplus F1)\oplus \Span\{x^ky^l\mid k,l\geq1,\,\alpha k< l\}$, $\alpha\in\mathbb{R}\setminus\mathbb{Q}$.

In cases IV)--VI), we have $\alpha>0$.
\end{theorem}

\begin{proof}
It is clear that all decompositions listed in Theorem hold. 
In the cases 1 and 2.1, we have decompositions I)--III).
In the case 2.2, the scalars $\alpha$ and $\beta$ are well-defined, 
and by Lemma~\ref{Sh=0_Lemma1}, Lemma~\ref{Sh=0_Lemma2}, and Corollary~\ref{Sh=0_Coro} we get decompositions IV)--VI). 
By Remark~\ref{Sh=0_remark2}, decompositions I)--VI) are pairwise distinct.
\end{proof}

Let us explain how the cases IV)--VI) differ from each other.
If a point $(k,l)\in \mathbb{N}^2 \setminus\{(0,0)\}$ 
from the line $y = \alpha x$ lies in one of the subalgebras, then $\alpha\in\mathbb{Q}$ and 
the set $C$ of all nonnegative integer points from the line lies entirely in one of the subalgebras.
If $C\subset A$, then it is the case V).
If $C\subset B$, then it is the case IV).
Finally, if $C\not\subset A$ and $C\not\subset B$, then $\alpha\in\mathbb{R}\setminus\mathbb{Q}$, and it is the case VI).

\begin{remark}
Let $R$ be an RB-operator of weight 1 on $F[x]$, then operator $P(x^n y^m) = R(x^n)y^m$ is an RB-operator of weight 1 on $F[x]$.
Applying this formula for operators from Theorem~\ref{F[x]-1}, we get cases I)--III) from Theorem~\ref{Thm:decompositions}.
\end{remark}

\section{RB-operators coming from Case (5)}

An RB-operator~$R$ defined by the homomorphic averaging operator~$T$ from Case~(5) of Theorem~\ref{AveOpClassif} maps $x$ and $y$ to the scalars.
Actually, such RB-operator is uniquely determinated by 
the values $R(x)$ and $R(y)$.

\begin{lemma}\label{x,y->const}
Let $R$ be a nonzero monomial RB-operator of weight~1 
such that $R(x) = \alpha$, $R(y) = \beta$, $\Imm R\subseteq F$, then $R = -T$, where~$T$ is defined in Case~(5) of Theorem~\ref{AveOpClassif}.
\end{lemma}

\begin{proof}
Let us assume that the weight of $R$ equals to $-1$.
From the equality
$$
R(x^n)R(y^m)=R(\alpha_{n,0}y^m + \alpha_{0,m} x^n - x^n y^m),
$$
we obtain $\alpha_{n,m} = \alpha_{n,0}\alpha_{0,m}$. 
It implies that for $n = 0$ we have $\alpha_{0,m} = \alpha_{0,0}\alpha_{0,m}$
and for $m = 0$ we have $\alpha_{n,0} = \alpha_{0,0}\alpha_{n,0}$, 
whence $\alpha_{0,0}=1$, otherwise $R = 0$.
Since $\alpha_{0,m} = \beta^m$ and $\alpha_{n,0} = \alpha^n$ (see the proof of Theorem~\ref{F[x]-1}~\cite{Monom}), we get the required form of the operator.
\end{proof}

\section{RB-operators coming from Case (1)}

\subsection{Case (1) for $r=0$}

Note that Cases (1) and (2) when $r = 0$ coincide.

\begin{lemma}\label{lemma_r=0}
Let $R$ be an RB-operator of weight~1 and $R(x^n y^m) = \alpha_{n,m} y^m$ for some $\alpha_{n,m}\in F$ such that $\Imm R\not\subseteq F1$.
Then there exist $\beta,\gamma\in F$ such that $R$ has the following form,
$$
R(x^n y^m) = \begin{cases}
-\beta^n y^m, & m>0, \\
\theta\gamma^n, & m=0, 
\end{cases}
$$
where $\theta = \{0,-1\}$.
\end{lemma}

\begin{proof}
Let us assume that the weight of $R$ equals to $-1$.
By the condition of Lemma, 
$y^k\in\Imm R$ for some $k>0$. If $R(y) = 0$, then $y^k\in \ker R$, a~contradiction with $\ker R\cap \Imm R = (0)$.
Hence, $R(y) = y$ and so, $R(y^l) = y^l$ for all $l>0$.
By~\eqref{RBO}, we get the relations on $\alpha_{n,m}$,
\begin{equation}\label{epsilon_r=0}
\alpha_{n,m}\alpha_{s,t} = 
\alpha_{n,m}\alpha_{s,t + m} + \alpha_{s,t}\alpha_{n,t + m} - \alpha_{n + s,t + m}.
\end{equation}
By Theorem~\ref{F[x]-1}, we have $\alpha_{n,0} = \gamma^n$ for some $\gamma\in F$.
By~\eqref{epsilon_r=0}, for $m = 0$ we have $\alpha_{n + s,t} = \alpha_{s,t}\alpha_{n,t}$.
Thus, $\alpha_{n,m} = \beta_{m}^n$ for $n,m >0$ and some $\beta_{m}\in F$.
Thereby, we rewrite \eqref{epsilon_r=0}~for $n = s = 1$ as follows,
$$
(\beta_m - \beta_{t+m})(\beta_t-\beta_{t+m}) = 0.
$$
Taking $m = t = 1$, we derive that $\beta_2 = \beta_1$.
Further, by induction one may prove that $\beta_i = \beta_1$ for every $i>1$.
\end{proof}


\subsection{Case (1) for $r\geq 1$}

\begin{lemma}
Let $R$ be a nonzero RB-operator of weight~1 
and $R(x^n y^m) = \alpha_{n,m} x^{rm} y^m$, $r\geq 1$, $\Imm R\neq F$.
Then $R$ has the following form
\begin{equation} \label{eq:Case1r>=1}
R(x^n y^m) 
 = \begin{cases}
-(-\alpha_{0, 1})^{\frac{rm - n}{r}}x^{r m} y^m, & m > 0, \\
\theta\alpha_{1, 0}^n, & m = 0, \\
\end{cases}
\end{equation}
where $\theta\in \{0, -1\}$.
\end{lemma}

\begin{proof}
By~\eqref{RBO}, we have
\begin{equation} \label{NewEpsilonSystem}
\alpha_{n, m} \alpha_{s, t} = \alpha_{n, m}\alpha_{rm + s, m + t}
+ \alpha_{s, t} \alpha_{tr + n, m + t} + \alpha_{n + s, m + t}. 
\end{equation}
As in the proof of Lemma~\ref{lemma_r=0}, we get $\alpha_{n, 0} = \theta\alpha_{1, 0}^n$, where $\theta\in\{0,-1\}$.

Let $m\geq1$ and denote $q_{n} := \alpha_{n, m}$.
By~\eqref{NewEpsilonSystem} for $t = 0$ we have 
$$
q_n\alpha_{s, 0} = q_{n}q_{rm + s} + \alpha_{s, 0} q_{n} + q_{n + s}, 
$$
i.\,e. $q_{n + s} = -q_{n}q_{rm + s}$.

{\sc Case 1}: $\alpha_{n, m} \neq 0$ for all $n\geq0$ and $m\geq1$.
If $s = 0$, then $q_{rm} = -1$.
For $s = 1$, $q_{n + 1} = -q_{n}q_{rm + 1}$. We may continue on the sequence of the equalities, expressing 
$$
q_{n + 1} = (-q_{rm + 1})^{n + 1} q_0.
$$
Taking $n = rm - 1$, we get $q_{rm} = (-q_{rm + 1})^{rm} q_0= - 1$, 
hence $q_{rm + 1} = - \sqrt[rm]{-q_0^{-1} }$,
this root can be expressed, otherwise $R$~is not an RB-operator 
of the required form satisfying the conditions of Lemma.
Thus,
$$
q_n 
 = q_0(-q_{rm + 1})^n
 = -(-q_0)(\sqrt[rm]{-1/q_0})^n
 = -(-q_0)^{\frac{rm - n}{rm}}.
$$
We may write down the following equality,
\begin{equation} \label{Case1r>0QnViaQ0}
\alpha_{n, m} = - (-\alpha_{0, m})^{\frac{rm - n}{rm}},\ m \geq 1, \quad 
\alpha_{n, 0} = -\alpha_{1, 0}^n.
\end{equation}
Let us show that these roots are expressed via $\alpha_{0,1}$. 
For this, denote $\beta_m = -1/\alpha_{rm+1,m}$, then~\eqref{Case1r>0QnViaQ0} implies the equality $\alpha_{n,m} = - \beta_m^{rm -n}$,
and~\eqref{NewEpsilonSystem} takes the form 
$$
\beta_m^{rm - n}\beta_t^{rt - s} =
\beta_m^{rm - n}\beta_{m + t}^{rt - s} + \beta_t^{rt - s}\beta_{m + t}^{rm - n}
- \beta_{m + t}^{r(m + t) - n - s}.
$$
Collecting the terms, we have 
$$
\big(\beta_m^{rm-n} - \beta_{m+t}^{rm-n}\big)\big(\beta_t^{rt-s}-\beta_{m+t}^{rt-s}\big) = 0.
$$
For $n = rm-1$ and $s = rt-1$ the following relation holds, $(\beta_m-\beta_{m+t})(\beta_t-\beta_{m+t}) = 0$.
It is easy to check that $\beta_m = \beta_1$ for all $m\geq1$ by induction on~$m$.
The relation $\beta_m = \beta_1$ implies 
$\alpha_{0, m} = - (-\alpha_{0, 1})^m$, and then
$$
-\alpha_{n, m} = (-\alpha_{0, m})^{\frac{rm - n}{rm}}  
 = ((-\alpha_{0, 1})^m)^{\frac{rm - n}{rm}}
 = (((-1 / \alpha_{0, 1})^m)^{1/rm})^{n-rm}.
$$ 
In the obtained expressions for different $m$, distinct roots of the degree $r$ from $-1/\alpha_{0,1}$ may occur.
Let this root be fixed for $m = 1$, we express other roots through it.
Introduce the multiplier $\varepsilon(m)$ which depends on~$m$, i.\,e. $\varepsilon(m)^r = 1$ and $\varepsilon(1) = 1$.
Then $-\alpha_{n, m} = (-\alpha_{0, 1})^{\frac{rm - n}{r}}\varepsilon(m)^n$.
Applying this equality for $\alpha_{n,m}$ in~\eqref{NewEpsilonSystem}, we get
$$
(\varepsilon(m)^n-\varepsilon(m+t)^n)
(\varepsilon(t)^s-\varepsilon(m+t)^s) = 0, \quad m,t\geq1. 
$$
For $n = s = 1$, we conclude that $\varepsilon(m) = 1$ for all $m\geq1$.

{\sc Case 2}:
$\alpha_{n,m} = 0$ for some $n\geq0$, $m\geq1$.
First, let us show that $\alpha_{d,m} = 0$ for all $d\geq0$.
Let us again apply the notation $q_n = \alpha_{n,m}$.
Suppose to the contrary, there exists $f$ such that $q_f\neq0$.
The above stated equality $q_{d+s} = -q_d q_{rm+s}$ for $s=0$ and $d=f$ implies  $q_{rm} = -1$.
Next, $q_{rm} = (-q_{rm+1})^{rm}q_0$, hence $q_0,q_{rm+1}\neq0$. 
Therefore, $q_n = (-q_{rm+1})^nq_0\neq0$ for any $n$, it is a~contradiction.

Now we prove that $\alpha_{s, t} = 0$ for all $s\geq 0$, $t\geq1$.
Suppose to the contrary, $\alpha_{0,t}\neq0$ for some $t\geq1$.
Then by~\eqref{NewEpsilonSystem} we have
$\alpha_{0,t}^2 = 2\alpha_{0,t}\alpha_{rt,2t} + \alpha_{0,2t}$.
We conclude that $\alpha_{0,2t}\neq0$, otherwise $\alpha_{rt,2t} = 0$ by the previous paragraph, and so, $\alpha_{0,t}^2 = 0$, a contradiction.
Analogously, we get $\alpha_{0,kt}\neq0$ for all $k\geq1$.
Since $\alpha_{0,m} = 0$,
we derive $\alpha_{0,km}=0$ for all $k\geq1$ by~\eqref{NewEpsilonSystem}. It means that $\alpha_{0,mt}$ has to be zero and nonzero, we arrive at a~contradiction.

Therefore, all coefficients $\alpha_{a,b}$ for $b\geq1$ are simultaneously zero or nonzero.
If $\alpha_{a,b} = 0$ for all $b>0$, then $\Imm R = F$, it contradicts to the conditions of Lemma.
\end{proof}

\begin{remark}
In~\eqref{eq:Case1r>=1}, 
the expression $(-\alpha_{0, 1})^{\frac{rm - n}{r}}$ means
$(-\alpha_{0, 1})^m \times(\sqrt[r]{-1/\alpha_{0, 1}})^n$
for some root of degree $r$ from $-1/\alpha_{0,1}$.
\end{remark}

\section{RB-operators coming from Case (2) for $r=1$}\label{sec:Case2r=1}

Let an RB-operator $R$ have the form $R(x^n y^m) = \alpha_{n,m} y^{n + m}$ and $\Imm R\not\subseteq F1$. 
By~\eqref{RBO}, we get the following system of equations on the coefficients  
\begin{equation}\label{epsilon_r=1}
\alpha_{n,m}\alpha_{s,t} = 
\alpha_{n,m}\alpha_{s,t + m + n} + \alpha_{s,t}\alpha_{n,t + m + s} + \alpha_{n + s,t + m}.
\end{equation}

As in the proof of Lemma~\ref{lemma_r=0}, we conclude that $R(y^l) = -y^l$ for all $l>0$, i.\,e. $\alpha_{0,l} = -1$.
Thus, for $n=0$ or $s=0$ we get in~\eqref{epsilon_r=1} the trivial equality $0 = 0$.

We express $\alpha_{1,m}$ for $m > 2$ from the following equations
\begin{gather*}\label{system1011}
\alpha_{1,0}\alpha_{1,m-1}
 = \alpha_{1,0}\alpha_{1,m} + \alpha_{1,m-1}\alpha_{1,m} + \alpha_{2,m - 1}, \\
\alpha_{1,1}\alpha_{1,m-2}
 = \alpha_{1,1}\alpha_{1,m} + \alpha_{1,m-2}\alpha_{1,m} + \alpha_{2,m - 1}.
\end{gather*}
Subtracting the second equation from the first one, we obtain
\begin{gather}\label{1_1n}
(\alpha_{1,0} + \alpha_{1,m-1} - \alpha_{1,1} - \alpha_{1,m-2})\alpha_{1,m} = \alpha_{1,0}\alpha_{1,m-1} - \alpha_{1,1}\alpha_{1,m-2}. 
\end{gather}
If the denominator in~\eqref{1_1n} is nonzero, then we express $\alpha_{1,m}$ as follows,
\begin{gather}\label{2_1n}
\alpha_{1,m} = \frac{\alpha_{1,0}\alpha_{1,m-1} - \alpha_{1,1}\alpha_{1,m-2}}{\alpha_{1,0} + \alpha_{1,m-1} - (\alpha_{1,1} + \alpha_{1,m-2})}. 
\end{gather}
Analogously, we get a more general formula for $0 \leq t,k \leq m - 1$, $m>2$,
\begin{gather}\label{1general_alpha1m}
(\alpha_{1,k} + \alpha_{1,m-k-1} - \alpha_{1,t} - \alpha_{1,m-t-1})\alpha_{1,m} = \alpha_{1,k}\alpha_{1,m-k-1} - \alpha_{1,t}\alpha_{1,m-t-1}.
\end{gather}
Note that the formula~\eqref{1general_alpha1m} still holds when $m=1,2$.
Provided the denominator is nonzero, we obtain by~\eqref{1general_alpha1m} the formula
\begin{gather}\label{2general_alpha1m}
\alpha_{1,m} = \frac{\alpha_{1,k}\alpha_{1,m-k-1} - \alpha_{1,t}\alpha_{1,m-t-1}}{\alpha_{1,k} + \alpha_{1,m-k-1} - (\alpha_{1,t} + \alpha_{1,m-t-1})}. 
\end{gather}

By~\eqref{epsilon_r=1}, for $s = 1$ and $t=0$ we have
\begin{equation}\label{k+1n}
\alpha_{k + 1,m} = \alpha_{k,m}(\alpha_{1,0}-\alpha_{1,k+m}) - \alpha_{1,0}\alpha_{k,m + 1}.
\end{equation}

Thus, $R$ is determinated by the coefficients $\alpha_{s,t}$, where $s = 1$ and $t$ is arbitrary.
Our current goal is to find parameters through which all the coefficients would be expressed.
Afterwards, we will see that the set of such parameters is finite.

\subsection{Case $\alpha_{1,0} = \alpha_{1,1}$}

\begin{lemma} \label{lem:q0=q1}
Fix $\alpha_{1,0}\in F$ and let $\alpha_{1,i} = \alpha_{1,0}$ for all $i\in\mathbb{N}$.
Then an RB-operator $R$ takes the form 
$R(x^n y^m)= (-1)^{n+1}\alpha_{1,0}^n y^{n+m}$. 
\end{lemma}

\begin{proof}
We have to express $\alpha_{n,m}$ for arbitrary $n > 1$.
By~\eqref{k+1n} for $k = 1$ we have
$$
\alpha_{2,m} = \alpha_{1,n}(\alpha_{1,0}-\alpha_{1,n+1})
- \alpha_{1,0}\alpha_{1,n + 1} = -\alpha_{1,0}^2.
$$
Further, suppose that for $k < n + 1$ the equality $\alpha_{n,m} = (-1)^{n+1}\alpha_{1,0}^n$ is fulfilled, then we obtain by induction
\begin{multline*}
\alpha_{n + 1,m} 
 = \alpha_{n,m}(\alpha_{1,0}-\alpha_{1,n+m}) 
 - \alpha_{1,0}\alpha_{n,m + 1} = -\alpha_{1,0}\alpha_{n,m + 1}  \\
 = (-\alpha_{1,0})(-1)^{n+1}\alpha_{1,0}^n 
 = (-1)^n\alpha_{1,0}^{n+1}.
 \qedhere 
\end{multline*}
\end{proof}

\begin{definition}
Let $P$ be a~condition, then the function $\chi_P$ is called characteristic if $\chi_P = 1$, when $P$ is true, and $\chi_P = 0$, otherwise.

\end{definition}

For example, $\chi_{n+1\,|\,k+m} = 1$ if $n+1$ divides $k+m$, and $\chi_{n+1\,|\,k+m} = 0$, otherwise.

\begin{lemma} \label{lem:q0,qn}
Let $\alpha_{1,0}=\ldots=\alpha_{1,n-1}$, $n>1$, and $\alpha_{1,n}\neq \alpha_{1,0}$, then
$$
R(x^ky^m)
 = \begin{cases}
  {-}\!\left(\!(-\alpha_{1,0})^k
  {+} (-\alpha_{1,0})^{k-1}\dfrac{k(n+1)}{m+k}(\alpha_{1,0}{-}\alpha_{1,n})\chi_{n+1\,|\,k+m}\!\right)\!y^{k+m}, & k>0, \\
 -y^m, & k {=} 0, m>0, \\
 \theta, & k = m = 0,
  \end{cases}
$$
where $\theta\in\{0,-1\}$.
\end{lemma}

\begin{proof}
We prove the following formula
\begin{equation} \label{a...ab-formula}
\alpha_{k,m}
 = -(-\alpha_{1,0})^k
 - (-\alpha_{1,0})^{k-1}\frac{k(n+1)}{m+k}(\alpha_{1,0}-\alpha_{1,n})\chi_{n+1\,|\,k+m}, \quad k\geq1,
\end{equation}
by induction on $k$.
For $k = 1$, let us show the following equality 
\begin{equation} \label{a...ab-formula-start}
\alpha_{1,m} = \alpha_{1,0} - \frac{n+1}{m+1}(\alpha_{1,0}-\alpha_{1,n})\chi_{n+1\,|\,m+1}
\end{equation}
by induction on~$m$.
For $m = 0,\ldots,n$, \eqref{a...ab-formula-start}
holds trivially.
Let $m\geq n+1$ and suppose that~\eqref{a...ab-formula-start} is fulfilled for all $t<m$.
Let $n+1\not|m+1$, so, $\alpha_{1,m-n-1} = \alpha_{1,0}$.
If also $n+1\not|m$, so, $\alpha_{1,m-1} = \alpha_{1,0}$,
then by~\eqref{2general_alpha1m} and by the induction hypothesis we obtain
$$
\alpha_{1,m} 
 = \frac{\alpha_{1,m-n-1}\alpha_{1,n} 
 - \alpha_{1,0}\alpha_{1,m-1}}{\alpha_{1,m-n-1} + \alpha_{1,n} - (\alpha_{1,0} + \alpha_{1,m-1})}
 = \frac{\alpha_{1,0}\alpha_{1,n}-\alpha_{1,0}\alpha_{1,m-1}}{\alpha_{1,n}-\alpha_{1,m-1}}
 = \alpha_{1,0}.
$$
Let $n+1\,|\,m$, then analogously 
$$
\alpha_{1,m} 
 = \frac{\alpha_{1,m-n-1}\alpha_{1,n} 
 - \alpha_{1,1}\alpha_{1,m-2}}{\alpha_{1,m-n-1} + \alpha_{1,n} - (\alpha_{1,1} + \alpha_{1,m-2})}
 = \frac{\alpha_{1,0}(\alpha_{1,n}-\alpha_{1,m-2})}{\alpha_{1,n}-\alpha_{1,m-2}}
 = \alpha_{1,0}.
$$
Now let $n+1\,|\,m+1$, so, $\alpha_{1,m-n-1}\neq\alpha_{1,0}$ and $\alpha_{1,m-1} = \alpha_{1,0}$, since $n+1\not| m$. 
We calculate
\begin{multline*}
\alpha_{1,m} 
 = \frac{\alpha_{1,m-n-1}\alpha_{1,n} 
 - \alpha_{1,0}\alpha_{1,m-1}}{\alpha_{1,m-n-1} + \alpha_{1,n} - (\alpha_{1,0} + \alpha_{1,m-1})}
 = \frac{\alpha_{1,n}\big(\alpha_{1,0}-\frac{n+1}{m-n}(\alpha_{1,0}-\alpha_{1,n})\big)-\alpha_{1,0}^2}{\alpha_{1,0}-\frac{n+1}{m-n}(\alpha_{1,0}-\alpha_{1,n})+\alpha_{1,n}-2\alpha_{1,0}} \\
 = \frac{\alpha_{1,n}((m-2n-1)\alpha_{1,0}+\alpha_{1,n}(n+1))-\alpha_{1,0}^2(m-n)}{(m+1)(\alpha_{1,n}-\alpha_{1,0})}
 = \alpha_{1,0} - \frac{n+1}{m+1}(\alpha_{1,0}-\alpha_{1,n}).
\end{multline*}

Let the equality~\eqref{a...ab-formula} be proven for~$k\geq1$, we prove it for~$k+1$.
Denote $A(k,m) = \frac{k(n+1)}{m+k}(\alpha_{1,0}-\alpha_{1,n})\chi_{n+1\,|\,k+m}$.
By~\eqref{k+1n}, by the induction hypothesis and by $\chi_{n+1\,|\,k+m}\cdot\chi_{n+1\,|\,k+m+1} = 0$, we calculate 
\begin{multline*}
\alpha_{k+1,m} 
 = \alpha_{k,m}(\alpha_{1,0} - \alpha_{1,m+k}) - \alpha_{1,0}\alpha_{k,m+1} \\
 = ( -(-\alpha_{1,0})^k 
 - (-\alpha_{1,0})^{k-1}A(k,m))\left(\frac{n+1}{m+k+1}(\alpha_{1,0}-\alpha_{1,n})\chi_{n+1\,|\,k+m+1}\right)  \\
 - \alpha_{1,0}( -(-\alpha_{1,0})^k - (-\alpha_{1,0})^{k-1}A(k,m+1)) 
 = -(-\alpha_{1,0})^{k+1}  \\
 + (-\alpha_{1,0})^k\left[
 - \frac{n+1}{m+k+1}(\alpha_{1,0}-\alpha_{1,n})\chi_{n+1\,|\,k+m+1}
 - \frac{k(n+1)}{m+k+1}(\alpha_{1,0}-\alpha_{1,n})\chi_{n+1\,|\,k+m+1}
 \right] \\
 = -(-\alpha_{1,0})^{k+1} 
 - (-\alpha_{1,0})^k\frac{(k+1)(n+1)}{m+k+1}(\alpha_{1,0}-\alpha_{1,n})\chi_{n+1\,|\,k+m+1},
\end{multline*}
as it is required.
\end{proof}

\begin{example}\label{exm:r=1q0=q1}
Under the conditions of Lemma~\ref{lem:q0,qn} with $\alpha_{1,0} = 0$ and $\alpha_{1,n}\neq0$, we get the operator
$$
R(x^s y^m) = 
\begin{cases}
\frac{\alpha_{1,n}}{d} y^{m + 1}, & s = 1, m = d(n+1) - 1,\ d\geq1, \\
- y^m, & s = 0, m > 0, \\
\theta, & s = m = 0, \\
0, & \mbox{otherwise},
\end{cases}
$$
where $\theta \in \{0,-1\}$. 
\end{example}

\begin{remark}
Let $R$ be a~monomial RB-operator of nonzero weight on $F[x,y]$.
The condition $t^k\in \ker R$ does not imply that $t\in\ker R$.
Indeed, let $R$ be the operator from Example~\ref{exm:r=1q0=q1}, where $t = x y^m$ for $m = d(n+1) - 1$ and some~$d\in\mathbb{N}_{>0}$.
For $k > 1$ we have $R(t^k) = 0$, as $t^k=x^k y^{m k}\in\ker R$.
\end{remark}

\subsection{Case $\alpha_{1,0} = \alpha_{1,2}$}

\begin{lemma}  \label{lem:q0=q2}
Let $\alpha_{1,0} = \alpha_{1,2} \neq \alpha_{1,1}$, then an RB-operator~$R$ takes the following form
\begin{equation} \label{aba-formula}
R(x^k y^n)
 = (-1)^{k+1} 
 \alpha_{1,0}^{k- (n - k - 1 \!\!\!\!\pmod{2})}
 \left( \frac{(n-k)\alpha_{1,0} + 2k{\alpha_{1,1}}}{n + k} \right)^{n - k - 1 \!\!\!\!\pmod{2}}
 y^{k+n}.
\end{equation}
\end{lemma}

\begin{proof}
Let us show that $\alpha_{1,2n} =\alpha_{1,0}$ and
$\alpha_{1,2n + 1} = \frac{n \alpha_{1,0} + \alpha_{1,1}}{n + 1}$.
We prove these formulas by induction on~$n$.
The induction base is trivial, for the induction step, we write down by~\eqref{2_1n}
for $n>0$
$$
\alpha_{1,2n + 2} = \frac{\alpha_{1,0}\alpha_{1,2n + 1} - \alpha_{1,1}\alpha_{1,2n }}{\alpha_{1,0} + \alpha_{1,2n + 1} - \alpha_{1,1} - \alpha_{1,2n}} =
\frac{\alpha_{1,0}(\alpha_{1,2n + 1} - \alpha_{1,1})}{\alpha_{1,2n + 1} - \alpha_{1,1}}.
$$
Since $\alpha_{1,2n + 1} - \alpha_{1,1} \neq 0$, we get the first formula.
We apply it to prove the second formula as follows,
$$
\alpha_{1,2n + 3} = \frac{\alpha_{1,0}\alpha_{1,2n + 2} - \alpha_{1,1}\alpha_{1,2n + 1}}{\alpha_{1,0} + \alpha_{1,2n + 2} - \alpha_{1,1} - \alpha_{1,2n + 1}}
 = \frac{(n+1)\alpha_{1,0}^2 - n\alpha_{1,0}\alpha_{1,1} - \alpha_{1,1}^2}{(2n + 2)\alpha_{1,0} - (n + 1)\alpha_{1,1} - n\alpha_{1,0} - \alpha_{1,1}}.
$$
Note that
$(n+1)\alpha_{1,0}^2 - n\alpha_{1,0}\alpha_{1,1} - \alpha_{1,1}^2 = ((n+1)\alpha_{1,0} 
 + \alpha_{1,1})(\alpha_{1,0} - \alpha_{1,1})$, 
then 
$$
\alpha_{1,2n + 3} = \frac{(n+1)\alpha_{1,0} + \alpha_{1,1}}{(n + 1) + 1}.
$$
We prove that $R$ takes the form~\eqref{aba-formula} by induction on $k$.
Let $n+1-k$ be even.
Applying~\eqref{k+1n} and the induction hypothesis, we get
\begin{multline*}
\alpha_{k + 1,n} 
 = (-1)^{k+1} \alpha_{1,0}^{k+1} 
 {-} (-1)^{k+1}\alpha_{1,0}^k
 \left( \frac{ (n - 1 + k)\alpha_{1,0} + 2 \alpha_{1,1}}{n + k + 1} \right) \\
 - (-1)^{k+1} \alpha_{1,0}^k \left( \frac{ (n + 1- k)\alpha_{1,0} + 2k\alpha_{1,1}}{n + k + 1} \right) \\
 = (-1)^{k+1} \alpha_{1,0}^k\left(\frac{(n + k + 1)\alpha_{1,0}}{n + k + 1} - \frac{(n -1 + k)\alpha_{1,0} + 2\alpha_{1,1}}{n + k + 1} - \frac{(n + 1 - k)\alpha_{1,0} + 2k\alpha_{1,1}}{n + k + 1}\right) \\
 = (-1)^{(k+1)+1} \alpha_{1,0}^k\left( \frac{(n - (k + 1))\alpha_{1,0} + 2(k + 1)\alpha_{1,1}}{n + (k + 1)} \right).
\end{multline*}
In the case when $n+1-k$ is odd, the equality~\eqref{k+1n} implies 
$$
\alpha_{k + 1,n + 1} = - \alpha_{1, 0}\alpha_{k,n + 2}
= - \alpha_{1, 0}(-1)^{k+1} \alpha_{1,0}^k = (-1)^{k+2}\alpha_{1,0}^{k+1}.
 \qedhere
$$
\end{proof}

Note that RB-operator~\eqref{aba-formula} with $\alpha_{1,0} = 0$
coincides with the RB-operator from Example~\ref{exm:r=1q0=q1} with $n = 1$.

\subsection{Case $\alpha_{1,1} = \alpha_{1,2}$}

\begin{lemma}  \label{lem:q1=q2}
Let $\alpha_{1,0} \neq\alpha_{1,1} = \alpha_{1,2}$, then an RB-operator takes the following form
$$
R(x^n y^m) = \begin{cases}
-(-\alpha_{1,1})^n y^{n+m}, & n\geq2, m=0,\ n\geq1,m>0, \\
-y^m, & n = 0, m > 0, \\
\alpha_{1,0} y, & n = 1, m = 0, \\
\theta, & n=m=0,
\end{cases}
$$
where $\theta\in\{0,-1\}$.
\end{lemma}

\begin{proof}
Let us prove by induction on~$m\geq2$ that $\alpha_{1,m} = \alpha_{1,1}$.
The induction base is trivial.
Suppose that for all $k < m$ the induction hypothesis holds, then by~\eqref{2_1n}:
$$
\alpha_{1,m} = \frac{\alpha_{1,m - 1}(\alpha_{1,0} - \alpha_{1,1})}{\alpha_{1,0} - \alpha_{1,1}},
$$
therefore $\alpha_{1,m} = \alpha_{1,m - 1} = \alpha_{1,1}$, since $\alpha_{1,0} - \alpha_{1,1} \neq 0$.
We prove by induction on~$k$ that $\alpha_{k, m} = -(-\alpha_{1,1})^k$.
The base case $k=2$ follows from the equalities~\eqref{k+1n},
$$
\alpha_{2,m} = \alpha_{1,1}(\alpha_{1,0}-\alpha_{1,1}) - \alpha_{1,0}\alpha_{1,1} = -(-\alpha_{1,1})^2.
$$
Now we prove the induction step by~\eqref{k+1n},
$$
\alpha_{k + 1,m} = -(-\alpha_{1,1})^k(\alpha_{1,0}-\alpha_{1,1}) + \alpha_{1,0}(-\alpha_{1,1})^k
= \alpha_{1,1}(-\alpha_{1,1})^k = -(-\alpha_{1,1})^{k + 1}.
 \qedhere
$$
\end{proof}

\subsection{General case for $r = 1$}

First, we derive formulas for $\alpha_{1,i}$, which are interesting by themselves. 
Moreover, these formulas will be useful for proving that a monomial operator constructed due to them is a Rota---Baxter one.

Let $\alpha_{1,0}\neq \alpha_{1,1}$, $\alpha_{1,0}\neq \alpha_{1,2}$, and $\alpha_{1,1}\neq \alpha_{1,2}$,
since Lemmas~\ref{lem:q0=q1}--\ref{lem:q1=q2}
cover all subcases $\alpha_{1,0} = \alpha_{1,1}$, $\alpha_{1,0} = \alpha_{1,2}$, and
$\alpha_{1,1} = \alpha_{1,2}$. Denote
$$
\alpha = \frac{\alpha_{1,1}}{\alpha_{1,0}-\alpha_{1,1}},\quad
\beta = \frac{1}{\alpha_{1,0}-\alpha_{1,1}},\quad
s_i = \alpha - \beta \alpha_{1,i}.
$$

The formula~\eqref{1_1n} in terms of~$s_i$ has the following form
\begin{equation} \label{sm+1-Start}
s_{m+1}(s_{m-1}-s_m+1) = s_m, \quad m\geq2.
\end{equation}

We describe the first values of $s_k$:
\begin{equation} \label{small-si}
s_0 = -1, \quad
s_1 = 0, \quad
s_2 = s_2, \quad
s_3 = \frac{s_2}{1-s_2}.
\end{equation}

Due to the conditions on $\alpha_{1,i}$, $i=0,1,2$,
we have $s_2\neq-1$ and $s_2\neq0$.

\begin{remark} \label{rem:s2}
The last equality in~\eqref{small-si} is correct, as $s_2\neq1$, otherwise we get by~\eqref{sm+1-Start} that $s_2 = 0$, it contradicts with the initial assumption.
\end{remark}

Show that the equality $\alpha_{1,0} = \alpha_{1,i}$ fulfilled for some $i>0$ leads to the cases considered above.

\begin{lemma} \label{lem:2.9}
Let $\alpha_{1,0}\neq\alpha_{1,1},\alpha_{1,2}$.
Then $\alpha_{1,0} = \alpha_{1,i}$ is not fulfilled for every $i>0$.
\end{lemma}

\begin{proof}
To the contrary, let $m>2$ be a minimal number such that $\alpha_{1,0} = \alpha_{1,m}$, then by~\eqref{1_1n}
$$
\alpha_{1,0}^2 + \alpha_{1,0}\alpha_{1,m-1} - \alpha_{1,0}\alpha_{1,1} - \alpha_{1,0}\alpha_{1,m-2} = \alpha_{1,0}\alpha_{1,m-1} - \alpha_{1,1}\alpha_{1,m-2},
$$
which is equivalent to
$(\alpha_{1,0} - \alpha_{1,m - 2})(\alpha_{1,0} - \alpha_{1,1}) = 0$,
a~contradiction with the choice of~$m$.
\end{proof}

\begin{corollary}
Let $s_2\neq0,-1$. 
Then $s_{m-1}-s_m+1 = 0$ does not hold for each $m\geq2$.
\end{corollary}

\begin{proof}
For $m=2$ the required relation is not fulfilled, due to $s_2\neq-1$.
Let $m\geq3$ be a~minimal number such that $s_{m-1}-s_m+1 = 0$.
Then by~\eqref{sm+1-Start} we get $s_m = 0$ and $s_{m-1} = -1$.
Hence, $\alpha_{1,0} = \alpha_{1,m-1}$.
If $m = 3$, then $s_2 = -1$, it is a~contradiction with $s_2\neq -1$.
If $m > 3$, then the assumption contradicts to Lemma~\ref{lem:2.9}.
\end{proof}

Thus, we can rewrite the formula~\eqref{sm+1-Start} as
\begin{equation} \label{sm+1}
s_{m+1} = \frac{s_m}{s_{m-1}-s_m+1}, \quad m\geq2.
\end{equation}

\begin{lemma} \label{lm:sm+1Viasm}
Let $s_2\neq0,-1$.
Then for all $m\geq 0$ the following formula holds
\begin{equation} \label{sm-Rec-Start}
s_{m+1}(1/s_2 - s_m) = s_m + 1.
\end{equation}
\end{lemma}

\begin{proof}
We prove the statement by induction on~$m$.
For $m=0,1$, it follows directly.
For $m = 2$ we calculate
$$
\frac{s_2+1}{1/s_2-s_2}
 = \frac{s_2(s_2+1)}{1-s_2^2}
 = \frac{s_2}{1-s_2}
 = s_3.
$$

The following equalities applying~\eqref{sm+1} and the induction hypothesis prove~\eqref{sm-Rec-Start} for $m\geq3$:
\begin{multline*}
s_{m+1}(1/s_2-s_m) - s_m - 1
 = \frac{s_m(1/s_2-s_m)}{s_{m-1}-s_m+1} - s_m - 1 \\
 = \frac{s_m(1/s_2-s_m) 
 - (s_m+1)(s_{m-1}-s_m+1)}{s_{m-1}-s_m+1} 
 =\frac{s_m(1/s_2-s_{m-1})-s_{m-1}-1}{s_{m-1}-s_m+1}
 = 0.
 \qedhere
\end{multline*}
\end{proof}

\begin{lemma}\label{lem:smNew1/s2}
Let $s_2\neq0,-1$.
Then $s_m = 1/s_2$ does not hold for every $m\geq 0$.
\end{lemma}

\begin{proof}
Assume $s_m = 1/s_2$, then $s_m = -1$ by~\eqref{sm-Rec-Start}, so $s_2 = -1$, a contradiction.
\end{proof}

Due to Lemma~\ref{lem:smNew1/s2}, we may rewrite~\eqref{sm-Rec-Start} as follows,
\begin{equation} \label{sm-Rec}
s_{m+1} = \frac{s_m + 1}{1/s_2 - s_m}.
\end{equation}

Note that the relations~\eqref{sm-Rec} imply that
$$
\alpha_{1,m+1} = \frac{\alpha_{1,0}\alpha_{1,m} + A}{\alpha_{1,m} + B},
$$
where 
$B = \frac{\alpha_{1,0}^2 - 2\alpha_{1,0}\alpha_{1,1} + \alpha_{1,1}\alpha_{1,2}}{\alpha_{1,1}-\alpha_{1,2}}$, 
$A = \alpha_{1,1}B +\alpha_{1,0}(\alpha_{1,1}-\alpha_{1,0})$.

Under the change $y_m = s_m - 1/s_2$,
the recurrence~\eqref{sm-Rec} takes the form
\begin{equation} \label{ym-Rec}
y_{m+1} = -\frac{(1+s_2)}{s_2}\left(1 + \frac{1}{y_m}\right), \quad m\geq2.
\end{equation}
Let $y_m = x_{m+1}/x_m$, $m\geq2$, then by~\eqref{ym-Rec} we get the following linear recurrence relation of the second order
\begin{equation} \label{xm-Rec}
s_2x_{m+2} + (1+s_2)(x_{m+1} + x_m) = 0.
\end{equation}
The roots of the corresponding characteristic polynomial
\begin{equation} \label{s2-char}
s_2\lambda^2 + (1+s_2)(\lambda + 1)
\end{equation}
equal to
$\lambda_{\pm} = \frac{-1-s_2\pm\sqrt{1-2s_2-3s_2^2}}{2s_2}$.
Let $s_2\neq-1/3$ and let $F$ be a quadratically closed field, 
then a~solution~\eqref{xm-Rec} has the form $x_m = A\lambda_+^{m-2} + B\lambda_-^{m-2}$, $m\geq2$, for some $A$ and $B$. 
We express
\begin{equation} \label{sm-answer}
s_m = \frac{1}{s_2} + \frac{A\lambda_+^{m-1} + B\lambda_-^{m-1}}{A\lambda_+^{m-2} + B\lambda_-^{m-2}}
= \frac{1}{s_2}\left(1 + \frac{1}{2}\frac{A\mu_+^{m-1} + B\mu_-^{m-1}}{A\mu_+^{m-2} + B\mu_-^{m-2}} \right),
\end{equation}
where $\mu_{\pm} = -1-s_2\pm\sqrt{1-2s_2-3s_2^2}$.
Note that it is enough to find $A, B$ up to a nonzero scalar.
For $m=2$ in~\eqref{sm-answer}, we get
$$
A = -1+s_2+2s_2^2 + \sqrt{1-2s_2-3s_2^2}, \quad
B =  1-s_2-2s_2^2 + \sqrt{1-2s_2-3s_2^2}.
$$ 

\begin{lemma}
Let $s_2\neq0,-1$.
Then for all $m\geq1$ and for each $0\leq k\leq m-1$, the following relation holds
\begin{equation}\label{SumOfskInv}
s_k s_{m-k-1} - 1 
 = s_m(s_k + s_{m-k-1} + 1 - 1/s_2).
\end{equation}
\end{lemma}

\begin{proof}
For a~fixed $m\geq2$, we do induction on $k$.
For $k = 0$, \eqref{SumOfskInv} coincides with the formula~\eqref{sm-Rec-Start}.
Suppose that the equality~\eqref{SumOfskInv} holds for all $t\leq k$.

If $k = m - 1$, then~\eqref{SumOfskInv} again coincides with~\eqref{sm-Rec-Start}.
Let $k < m - 1$. 
Due to~\eqref{sm-Rec} we obtain
\begin{equation} \label{sm-RecInverse}
s_{m-1} = \frac{s_m/s_2 - 1}{s_m+1}.
\end{equation}
Here $s_m\neq-1$, otherwise, $s_2 = s_m = -1$, a contradiction.

On the one hand, we express by~\eqref{sm-Rec},~\eqref{sm-RecInverse}
\begin{multline*}
s_{k+1}s_{m-k-2} - 1
= \frac{s_k + 1}{1/s_2 - s_k}\frac{s_{m-k-1}/s_2 - 1}{s_{m-k-1}+1} - 1 \\
= \frac{ s_k s_{m-k-1}/s_2 + s_{m-k-1}/s_2 - s_k - 1 
 - (-s_ks_{m-k-1}+s_{m-k-1}/s_2 - s_k+1/s_2) }{(1/s_2 - s_k)(s_{m-k-1}+1)} \\
= \frac{ (1+1/s_2)( s_k s_{m-k-1} - 1) }{(1/s_2 - s_k)(s_{m-k-1}+1)}. 
\end{multline*}
On the other hand,
\begin{multline*}
s_m( s_{k+1} + s_{m-k-2} + 1 - 1/s_2)
= s_m\left( \frac{s_k + 1}{1/s_2 - s_k} + \frac{s_{m-k-1}/s_2 - 1}{s_{m-k-1}+1} + 1 - 1/s_2 \right) \\
= \!\frac{s_m}{(1/s_2 - s_k)(s_{m-k-1}+1)}\!\left(\! s_ks_{m-k-1}+s_{m-k-1}+s_k+1 
- \frac{s_ks_{m-k-1}}{s_2}+\frac{s_{m-k-1}}{s_2^2}+s_k-\frac{1}{s_2} \right. 
\allowdisplaybreaks
 \\
\left. + \left(1-\frac{1}{s_2}\right)\left(-s_ks_{m-k-1}+\frac{s_{m-k-1}}{s_2} - s_k+\frac{1}{s_2}\right)\right) \\
= \frac{s_m}{(1/s_2 - s_k)(s_{m-k-1}+1)}(1+1/s_2)(s_k+s_{m-k-1}+1-1/s_2).
\end{multline*}
From the obtained relations and the induction hypothesis, the formula~\eqref{SumOfskInv} is proved for $k+1$.
Hence, this formula holds by induction for all $m\geq2$ and $0\leq k\leq m-1$.
\end{proof}

\begin{corollary} \label{coro:sm-gen}
For all $0\leq t\leq k\leq m - 1$, $m\geq1$, the following relations are fulfilled:
\begin{gather}\label{2general_sm}
(s_k + s_{m-k-1} - s_t - s_{m-t-1})s_m 
 = s_k s_{m-k-1} - s_ts_{m-t-1}.
\end{gather}
\end{corollary}

The equations~\eqref{1general_alpha1m} in terms of~$s_i$ are identical to the equations~\eqref{2general_sm}.

Now, let us prove the main result of the work.
We show that any monomial operator on $F[x,y]$ of the form $R(x^n y^m) = \alpha_{n,m}y^{n+m}$ may be uniquely extended by the derived formulas from the values $\alpha_{1,0},\alpha_{1,1},\alpha_{1,2}$ to an RB-operator.
For this, we additionally exclude the case $2\alpha_{1,1}\neq \alpha_{1,0}+\alpha_{1,2}$; for RB-operators it does not hold by Remark~\ref{rem:s2}.

\begin{theorem} \label{theorem_r=1}
Let $\alpha_{1,0},\alpha_{1,1},\alpha_{1,2}\in F$ be pairwise distinct such that $2\alpha_{1,1}\neq\alpha_{1,0}+\alpha_{1,2}$.
Then a monomial operator of the form $R(x^n y^m) = \alpha_{n,m}y^{n+m}$, where $\alpha_{n,m}$ are expressed via $\alpha_{1,0},\alpha_{1,1},\alpha_{1,2}$ by ~\eqref{2_1n},~\eqref{k+1n}, is an RB-operator of weight $1$ on $F[x,y]$.
\end{theorem}

\begin{proof}
Let us prove~\eqref{epsilon_r=1} by inductions on $n$ and $s$.
Let $\alpha_{1,i} := q_i$.
For $n = 1$,
we check the required equality directly,
\begin{multline*}
\alpha_{1,m}\alpha_{s,t+m+1}+\alpha_{s,t}\alpha_{1,t+m+s}+\alpha_{s+1,t + m}-\alpha_{1,m}\alpha_{s,t} \\
 \mathop{=}\limits^{\eqref{k+1n}} 
 \alpha_{1,m}\alpha_{s,t+m+1}+\alpha_{s,t}\alpha_{1,t+m+s}
 + \alpha_{s,t+m}(\alpha_{1,0}-\alpha_{1,s+m+t})
 - \alpha_{1,0}\alpha_{s,m+t+1}
 - \alpha_{1,m}\alpha_{s,t} \\
 = H(s,t,m):= \alpha_{s,t+m+1}(q_m - q_0)
 + \alpha_{s,t+m}(q_0 - q_{m+s+t})
 + \alpha_{s,t}(q_{m+t+s}-q_m).
\end{multline*}
Let us show that $H(s,t,m) = 0$ for any $s$.
For $s = 1$, 
$$
H(1,t,m)
 = \alpha_{1,t+m+1}(\alpha_{1,m}+\alpha_{1,t}-\alpha_{1,0}-\alpha_{1,t+m})
 + \alpha_{1,0}\alpha_{1,t+m} - \alpha_{1,m}\alpha_{1,t}
 \mathop{=}\limits^{\eqref{1general_alpha1m}} 0.
$$
If $t+m\geq2$, then we have $H(1,t,m) = 0$ by~\eqref{1general_alpha1m}, which hold by Corollary~\ref{coro:sm-gen}.
Otherwise, $H(1,t,m) = 0$ trivially.

Show the inductive step for~$s$.
The equalities
\begin{gather*}
\alpha_{s+1,t+m+1}(q_m-q_0) 
 \mathop{=}\limits^{\eqref{k+1n}} 
 \alpha_{s,t+m+1}(q_m-q_0)(q_0-q_{s+t+m+1})
 - q_0\alpha_{s,t+m+2}(q_m-q_0), \\
\alpha_{s+1,t+m}(q_0{-}q_{s+t+m+1}) 
  \mathop{=}\limits^{\eqref{k+1n}}
\alpha_{s,t+m}(q_0{-}q_{s+t+m+1})(q_0{-}q_{s+t+m})
 - q_0\alpha_{s,t+m+1}(q_0{-}q_{s+t+m+1}),\\
\alpha_{s+1,t}(q_{s+t+m+1}-q_m) 
  \mathop{=}\limits^{\eqref{k+1n}}
 \alpha_{s,t}(q_{s+t+m+1}-q_m)(q_0-q_{s+t}) 
 - q_0\alpha_{s,t+1}(q_{s+t+m+1}-q_m)
\end{gather*}
imply that
\begin{multline*}
H(s+1,t,m) 
 = -q_0H(s,t+1,m) 
 + \alpha_{s,t+m+1}(q_m-q_0)(q_0-q_{s+t+m+1}) \\
 + \alpha_{s,t+m}(q_0-q_{m+s+t+1})(q_0-q_{s+t+m})
 + \alpha_{s,t}(q_{s+t+m+1}-q_m)(q_0-q_{s+t}) \\
 = -q_0H(s,t+1,m) + (q_0-q_{s+t+m+1})H(s,t,m)  
 \allowdisplaybreaks
 \\
 + \alpha_{s,t}( q_{s+t+m+1}(q_0-q_m-q_{s+t}+q_{m+s+t}) + q_mq_{s+t} - q_0q_{s+t+m}) \\
 \mathop{=}\limits^{\eqref{1general_alpha1m}} 
 -q_0H(s,t+1,m) + (q_0-q_{t+s+m+1})H(s,t,m),
\end{multline*}
it is zero by the induction hypothesis.

Let $n>1$.
Let us prove that
$$
G(n,m,s,t)
 = \alpha_{n,m}\alpha_{s,t + m + n} + \alpha_{s,t}(\alpha_{n,t + m + s}- \alpha_{n,m}) + \alpha_{n + s,t + m}
$$
is equal to zero by induction on~$s$. 
For $s = 1$, we have $G(n,m,1,t) = H(n,m,t) = 0$.
Show the induction step for~$s$.
By the equalities
\begin{gather*}
\alpha_{n,m}\alpha_{s+1,t+m+n}
 \mathop{=}\limits^{\eqref{k+1n}} 
\alpha_{n,m}\alpha_{s,n+m+t}(q_0-q_{n+m+s+t})
 - q_0\alpha_{n,m}\alpha_{s,n+m+t+1}, \\
\alpha_{s+1,t}(\alpha_{n,m+s+t+1} {-} \alpha_{n,m})
  \!\mathop{=}\limits^{\eqref{k+1n}}\!
\alpha_{s,t}(\alpha_{n, m + s + t +1} {-} \alpha_{n,m})(q_0 {-} q_{s+t})
 {-} q_0\alpha_{s,t+1}(\alpha_{n,m+ s+ t+1} {-} \alpha_{n,m}), \\
\alpha_{n + s+1,m + t}
  \mathop{=}\limits^{\eqref{k+1n}}
 \alpha_{n + s,m + t}(q_0-q_{n+m+s+t})
 - q_0\alpha_{n+s,m+t+1},
\end{gather*}
we conclude that
\begin{multline*}
G(n,m,s+1,t) 
 = G(n,m,s,t)(q_0-q_{n+m+s+t}) - q_0G(n,m,s,t+1) \\
 + \alpha_{s,t}( \alpha_{n,m}(q_{s+t}-q_{n+m+s+t}) 
 - \alpha_{n,m+s+t}(q_0-q_{n+m+s+t})
 + \alpha_{n,m+s+t+1}(q_0-q_{s+t})) \\
 = G(n,m,s,t)(q_0-q_{n+m+s+t}) - q_0G(n,m,s,t+1)
 - \alpha_{s,t}H(n,m,s+t).
\end{multline*}
The final expression is equal to zero by the already proven equality $H(n,m,s+t) = 0$ and by the induction hypothesis.
\end{proof}

There are formulas that allow us to express $\alpha_{k,n}$ through the elements $\alpha_{1,i}$.

\begin{lemma}
Let $q_i = \alpha_{1,i}$, then
\begin{equation}\label{10Sum}
\alpha_{k + 1,m}
 = \sum\limits_{s = 0}^k (-1)^s q_0^{k - s} 
\left(
\sum\limits_{m \leq i_1 < \dots < i_{s + 1} \leq m + k} 
(-1)^{i_1 - m} \binom{k}{i_1 - m} q_{i_1} \dots q_{i_{s + 1}}
\right).
\end{equation}
\end{lemma}

\begin{proof}
Denote the right-hand side of~\eqref{10Sum} as $\beta_{k+1,m}$.
We prove that $\beta_{k+1,m} = \alpha_{k+1,m}$ by induction on $k$.
The induction base is true, since $\beta_{1,m} = q_m$.
Let us rewrite $\beta_{k+1,m}$ as follows,
\begin{multline*}
\!\!\!\!\beta_{k + 1,m} 
 {=} \sum\limits_{s = 0}^{k-1} ({-}1)^s q_0^{k - s} 
\bigg(
\sum\limits_{m \leq i_1 < \dots < i_{s + 1} \leq m + k} 
\!\!({-}1)^{i_1 - m} \binom{k}{i_1 - m} q_{i_1} \dots q_{i_{s + 1}} 
\bigg) {+} ({-}1)^k q_m \dots q_{m + k}
\\ 
 = \sum\limits_{s = 0}^{k - 1}(-1)^s q_0^{k - s} 
\bigg(
\sum\limits_{m \leq i_1 < \dots < i_{s + 1} \leq m + k} 
(-1)^{i_1 - m} \bigg( \binom{k - 1}{i_1 - m - 1} + \binom{k - 1}{i_1 - m} \bigg) q_{i_1} \dots q_{i_{s + 1}}
\bigg) \\
 + (-1)^k q_m \dots q_{m + k}.
\end{multline*}
We calculate the sums for $i_{s+1} = m+k$ and $i_{s+1} < m+k$ separately.
Note that changing the index $s+1$ by~$s$, we get the extra term $q_0^k q_{m+k}$, which we subtract in the following formula:
\begin{multline*}
\sum\limits_{s = 0}^{k - 1} (-1)^s q_0^{k - s} 
\bigg(
\sum\limits_{m \leq i_1 < \dots < i_{s + 1} \leq m + k} 
(-1)^{i_1 - m} \binom{k -1}{i_1 - m} q_{i_1} \dots q_{i_{s + 1}}
\bigg) \\
= q_0\sum\limits_{s = 0}^{k -1} (-1)^s q_0^{k - s - 1} 
\bigg(
\sum\limits_{m \leq i_1 < \dots < i_{s + 1} \leq m + k - 1}
(-1)^{i_1 - m} \binom{k -1}{i_1 - m}q_{i_1} \dots q_{i_{s + 1}}
\bigg) \allowdisplaybreaks \\ 
+ q_0q_{m + k}\sum\limits_{s = 0}^{k -1} (-1)^s q_0^{k - s - 1} 
\bigg(
\sum\limits_{m \leq i_1 < \dots < i_s \leq m + k - 1} 
(-1)^{i_1 - m} \binom{k -1}{i_1 - m} q_{i_1} \dots q_{i_s}
\bigg) - q_0^k q_{m+k} \allowdisplaybreaks \\
= q_0 \alpha_{k,m} 
+ q_{m + k}\sum\limits_{s = 1}^{k} (-1)^s q_0^{k - s} 
\bigg(
\sum\limits_{m \leq i_1 < \dots < i_s \leq m + k - 1} 
(-1)^{i_1 - m} \binom{k -1}{i_1 - m}q_{i_1} \dots q_{i_s}
\bigg) \\
 - (-1)^k q_m \dots q_{m + k} 
 = q_0\alpha_{k,m}-q_{k+m}\alpha_{k,m}
 - (-1)^k q_m \dots q_{m + k}.
\end{multline*}
Finally, we have
\begin{multline*}
\beta_{k + 1,m} 
 = -q_0\sum\limits_{s = 0}^{k - 1} (-1)^s q_0^{k - s - 1}\bigg(
 \sum\limits_{m+1 \leq i_1 < \dots < i_{s + 1} \leq m + k} 
 (-1)^{i_1 - m - 1} \binom{k - 1}{i_1 - m - 1}q_{i_1} \dots q_{i_{s + 1}}\bigg) \\
 + q_0\alpha_{k,m}-q_{k+m}\alpha_{k,m}
 - (-1)^k q_m \dots q_{m + k}
 + (-1)^k q_m \dots q_{m + k} \\
 = - q_0 \alpha_{k,m+1}+q_0\alpha_{k,m}-q_{k+m}\alpha_{k,m}
 = (\alpha_{1,0}-\alpha_{1,k+m})\alpha_{k,m} - \alpha_{1,0}\alpha_{k,m + 1}
 \mathop{=}\limits^{\eqref{k+1n}}\alpha_{k + 1,m}. 
\qedhere
\end{multline*}
\end{proof}

\subsection{Examples of RB-operators coming from averaging operators}

Now we consider a~very interesting example, when an RB-operator may be obtained by using the Fibonacci sequence.

For $m \in \mathbb{N}$, define the Fibonacci sequence $f_m$ as follows,
$$
f_m 
= \begin{cases}
0, & m = 0, \\
1, & m = 1, 2, \\
f_{m - 1} + f_{m - 2}, & m > 2.
\end{cases}
$$
This sequence can be generalized for $m \in \mathbb{Z}$ by the formula 
$f_{-m} := (-1)^{m + 1} f_m$, when $m > 0$.
Note that the relation $f_{i+2} = f_{i+1} + f_i$ holds for every $i \in \mathbb{Z}$.

\begin{example}
Let $a,b\in F$ be such that $a\neq b$. 
The following operator is an RB-operator of weight~1 on $F[x,y]$:
$$
R(x^n y^m) 
 = \begin{cases}
 \left(\frac{(-1)^{n-1}}{f_{m + n}} \sum\limits_{s = 0}^n 
    \binom{n}{s} a^{ k - s} b^s  f_{m - n + s}\right)y^{n+m}, & n+m>0, \\
  -1, & \mbox{otherwise}.
\end{cases}
$$
This RB-operator is determinated by the values $\alpha_{1,0} = a$, $\alpha_{1,1} = b$, and $\alpha_{1,2} = (a+b)/2$.
\end{example}

Let us prove the formula 
\begin{equation}\label{FIB_main}
\alpha_{k,m} = \frac{(-1)^{k - 1}}{f_{m + k}} \sum\limits_{s = 0}^{k} 
 \binom{k}{s} a^{ k - s} b^s  f_{m - k + s}.
\end{equation}
by induction on $m$ for $k = 1$ and afterwards by induction on $k$.
When $k=1$, the formula~\eqref{FIB_main} converts to the following one, 
\begin{gather}\label{FIB_k=1}
\alpha_{1,m} = \frac{f_{m - 1}a + f_m b}{f_{m + 1}}, \quad m \geq 0. 
\end{gather}
For $m = 0,1,2$, \eqref{FIB_k=1} holds due to the choice of $\alpha_{1,t}$, $t = 0,1,2$.
Suppose that~\eqref{FIB_k=1} is true for all $t < m$.
By the induction assumption and~\eqref{2_1n}, we express
\begin{multline*}
\alpha_{1, m + 1}
 = \frac{\alpha_{1,0}\alpha_{1,m} - \alpha_{1,1}\alpha_{1,m - 1}}{\alpha_{1,0} + \alpha_{1,m} - \alpha_{1,1} - \alpha_{1,m - 1}}  \\ 
 = \frac{ f_{m-1}f_m a^2 + (f_m^2 - f_{m - 2}f_{m+1})ab - f_{m-1}f_{m+1}b^2 }
 { (f_m f_{m+1} + f_m f_{m-1} - f_{m-2}f_{m+1})a - (f_m f_{m+1} + f_{m-1}f_{m+1} - f_m^2)b }.
\end{multline*}
Transform the coefficients of the denominator
\begin{gather*}
f_m f_{m+1} + f_m f_{m-1} - f_{m-2}f_{m+1} = 
f_{m-1}f_{m+1} + f_{m-1} f_m = f_{m+2}f_{m-1}, \\
f_m f_{m+1} + f_{m-1}f_{m+1} - f_m f_m = 
f_m f_{m-1} + f_{m+1}f_{m-1} = f_{m+2}f_{m-1}.
\end{gather*}
For the numerator, it is true that
$$
f_{m-1}f_m a^2 + (f_m^2 - f_{m - 2}f_{m+1})ab - f_{m-1}f_{m+1}b^2
 = f_{m-1}(a - b)(f_m a + f_{m+1}b).
$$
So~\eqref{FIB_k=1} is proven.

Suppose that~\eqref{FIB_main} holds for all $t < k$. 
By~\eqref{k+1n}, applying the induction assumption and already proven formula~\eqref{FIB_k=1}, we have
\begin{multline*}
\alpha_{k + 1, m} = \alpha_{k, m}( a - \alpha_{1, m + k}) - a\alpha_{k, m + 1} \\
 = \frac{(-1)^{k - 1}}{f_{m + k}} \sum\limits_{s = 0}^{k} 
 \binom{k}{s} a^{ k + 1 - s} b^s  f_{m - k + s} 
 +  \left(\frac{f_{m + k - 1}a + f_{m+k} b}{f_{m + k + 1}}\right)
  \frac{(-1)^k}{f_{m + k}} \sum\limits_{s = 0}^{k} \binom{k}{s} a^{ k - s} b^s  f_{m - k + s}  \\
  +  \frac{(-1)^k}{f_{m + k + 1 }} \sum\limits_{s = 0}^{k}\binom{k}{s} a^{ k + 1 - s} b^s  f_{m - k + s + 1}.
\end{multline*}
Compute the coefficient $\gamma_s$ at $a^{k + 1 - s}b^s$.

For $s = k + 1$, we have
$\gamma_{k+1} 
 = (-1)^k \binom{k}{k} \frac{f_{m + k}f_m}{f_{m + k}f_{m + k + 1}}
 = \frac{(-1)^k}{f_{m + k + 1}} \binom{k}{k}f_m$.

For $s = 0$, we have
$$
\gamma_0 
 = \frac{(-1)^k}{f_{m + k + 1}}\left( 
 - \frac{f_{m - k} f_{m + k + 1}}{f_{m + k}} 
 + \frac{f_{m + k -1}f_{m - k}}{f_{m + k}}
 + f_{m + 1 -k}\right).
$$ 
Since $f_{m + k - 1} f_{m - k} - f_{m + k + 1}f_{m - k} = -f_{m + k}f_{m - k}$, we obtain
$$
\gamma_0 
 = \frac{(-1)^k}{f_{m + k + 1}} (- f_{m - k} + f_{m - k + 1} )
 = \frac{(-1)^k}{f_{m + k + 1}} f_{m - k - 1}.
$$

Now we consider the case, when $ 1 \leq s \leq k$:
\begin{multline*}
\gamma_s
 = (-1)^k\left(
   - \binom{k}{s}\frac{f_{m - k + s}}{f_{m + k}}
    + \binom{k}{s} \frac{f_{m + k - 1}f_{m - k + s}}{f_{m + k}f_{m + k + 1}} \right. \\
 \left. + \binom{k}{s - 1}\frac{f_{m + k}f_{m - k + s - 1}}{f_{m + k}f_{m + k + 1}} 
    + \binom{k}{s}\frac{f_{m - k + s + 1}}{f_{m + k + 1}}
    \right) \\
    {=}  \frac{(-1)^k}{f_{m + k + 1}}\left(
    \binom{k}{s}\left(\frac{f_{m + k - 1}f_{m - k + s}}{f_{m + k}} 
    {-} \frac{f_{m - k + s}f_{m + k + 1}}{f_{m + k}}\right) 
    {+} \binom{k}{s - 1} f_{m - k + s - 1} 
    {+} \binom{k}{s} f_{m - k + s + 1}
    \right)\!.\!\!\!\!\!
\end{multline*}
The first two summands from the last formula are reduced to
$f_{m + k - 1}f_{m - k + s} - f_{m + k + 1}f_{m - k + s} = - f_{m + k}f_{m - k + s}$, 
so we can divide it by $f_{m + k}$.
Joint with the last summand, we get 
$\binom{k}{s} ( f_{m - k + s + 1} - f_{m - k + s}) = \binom{k}{s} f_{m - k + s - 1}$.
Finally, we have
\begin{gather*}
\gamma_s
 = \frac{(-1)^k}{f_{m + k + 1}}\bigg(
    \binom{k}{s} f_{m - k + s - 1} + \binom{k}{s -1} f_{m - k + s - 1}  \bigg) 
    =  \frac{(-1)^k}{f_{m + k + 1}} \binom{k + 1}{s}f_{m - k + s - 1}.
\end{gather*}
Therefore, we prove the induction step, and the operator~$R$ is Rota---Baxter one by Theorem~\ref{theorem_r=1}.

The appearance of the Fibonacci numbers~$f_n$ as denominators and coefficients at $\alpha_{1,0}$ and $\alpha_{1,1}$ is not surprising.
Indeed, we have $s_2 = -1/2$, and the characteristic polynomial~\eqref{s2-char} for the sequence $\{x_n\}_{n\geq2}$ is equal to $\lambda^2 -\lambda - 1$, this is exactly the characteristic polynomial for~$f_n$.

The construction of RB-operators through averaging operators was inspired by the following example.

\begin{example}[\!\!\cite{Viellard-Baron}]
The following operator is an RB-operator of weight $-1$ on $F[x,y]$:
$$
R(x^n y^m) 
 = \begin{cases}
\dfrac{m}{n+m}y^{n+m}, & n+m>0, \\
1, & \mbox{otherwise}.
\end{cases}
$$
This RB-operator is determinated by the values $\alpha_{1,0} = 0$, $\alpha_{1,1} = 1/2$, and $\alpha_{1,2} = 2/3$.
\end{example}

\section{Discussion}

We have not considered the case (2), when $r>1$.
Then we have the following system
\begin{equation}
\alpha_{n,m}\alpha_{s,t} = 
\alpha_{n,m}\alpha_{s,t + m + r n} + \alpha_{s,t}\alpha_{n,t + m + r s} + \alpha_{n + s,t + m}. \label{LastFormula}
\end{equation}
We have seen that for $r = 1$ the problem of describing RB-operators of the type (2) from Theorem~\ref{AveOpClassif} is not trivial. 
For $r > 1$, this problem is much more difficult.
We believe that all difficulties are concerned only with technical details and we expect that the approach applied for $r = 1$ works for $r>1$ too.
The main problem seems to be in consideration of all particular cases.

Let us show how obtained RB-operators of nonzero weight on $F[x,y]$ allow us to get RB-operators of nonzero weight on other algebras.

If one wants to describe monomial homomorphic averaging operators on the algebra of Laurent polynomials $F[x,x^{-1},y,y^{-1}]$, we get exactly the same list of such operators as in Theorem~\ref{AveOpClassif} with only one clarification: all $n,m,r$ are taken from $\mathbb{Z}$. Hence, we get plenty of nontrivial RB-operators of nonzero weight on $F[x,x^{-1},y,y^{-1}]$ from the already obtained RB-operators on $F[x,y]$.
Analogously to Problem~1, we state the following one.

{\bf Problem 2}. 
Describe all Rota---Baxter operators of nonzero weight on $F[x,x^{-1},y,y^{-1}]$ of the form 
$$
R(x^n y^m) = \alpha_{n,m}T(x^n y^m),\quad n,m\in\mathbb{Z},
$$
where $\alpha_{n,m}$ are some scalars from~$F$ 
and $T$ is a monomial homomorphic averaging operator on $F[x,x^{-1},y,y^{-1}]$.

Another direction of the applications of the obtained results is to consider monomial RB-operators on the algebra of formal power series $F[[x,y]]$.
Note that RB-operators arisen from averaging operators of Cases~(2) and~(3) are well-defined.

\section*{Acknowledgements}

The author is grateful to Vsevolod Gubarev for scientific supervision.
The author is also grateful to Maxim Goncharov and Valery G. Bardakov for the helpful remarks.

\noindent Artem Khodzitskii \\
Novosibirsk State University \\
Pirogova str. 2, 630090 Novosibirsk, Russia \\
e-mail: a.khodzitskii@g.nsu.ru
\end{document}